\documentclass[12pt]{amsart}

\calclayout

\let\oldtocsection=\tocsection

\let\oldtocsubsection=\tocsubsection

\let\oldtocsubsubsection=\tocsubsubsection

\renewcommand{\tocsection}[2]{\hspace{0em}\oldtocsection{#1}{#2}}
\renewcommand{\tocsubsection}[2]{\hspace{5em}\oldtocsubsection{#1}{#2}}
\renewcommand{\tocsubsubsection}[2]{\hspace{2em}\oldtocsubsubsection{#1}{#2}}

\usepackage{graphicx}
 \usepackage{amssymb}

\usepackage[colorlinks=true, pdfstartview=FitV, linkcolor=blue, citecolor=blue, urlcolor=blue]{hyperref}

\def\Xint#1{\mathchoice
{\XXint\displaystyle\textstyle{#1}}%
{\XXint\textstyle\scriptstyle{#1}}%
{\XXint\scriptstyle\scriptscriptstyle{#1}}%
{\XXint\scriptscriptstyle\scriptscriptstyle{#1}}%
\!\int}
\def\XXint#1#2#3{{\setbox0=\hbox{$#1{#2#3}{\int}$}
\vcenter{\hbox{$#2#3$}}\kern-.5\wd0}}

\def\avgint{\Xint-}

\newtheorem{theorem}{Theorem}[section]
\newtheorem{lemma}[theorem]{Lemma}

\newtheorem{definition}[theorem]{Definition}

\theoremstyle{definition}

\theoremstyle{remark}
\newtheorem{remark}[theorem]{Remark}

\numberwithin{equation}{section}
 \allowdisplaybreaks

\newcommand{\R}{\mathbb R}

\newcommand{\Z}{\mathbb Z}

\newcommand{\subRn}{{{\mathbb R}^n}}

\newcommand{\Q}{\mathcal Q}
\newcommand{\D}{\mathcal D}
\newcommand{\Ss}{\mathcal S}

\newcommand{\C}{\mathcal{C}}

\DeclareMathOperator{\supp}{supp}

\DeclareMathOperator*{\esssup}{ess\,sup}

\newcommand{\avg}[1]{\langle #1 \rangle}

\title[Elementary proofs of one weight inequalities]{Elementary proofs of one weight norm
  inequalities for fractional integral operators and commutators}

\author{David Cruz-Uribe, SFO}

\address{Department of Mathematics, Trinity
  College, Hartford, CT 06106, USA}
\email{david.cruzuribe@trincoll.edu}

\thanks{The author is supported by the Stewart-Dorwart faculty
  development fund at Trinity College and by NSF grant 1362425.}

\dedicatory{Dedicated to the memory of Professor Cora Sadosky}



\date{March 5, 2015}

\begin{document}

\maketitle

\begin{abstract}
  We give new and elementary proofs of one weight norm inequalities
  for fractional integral operators and commutators.  Our proofs are
  based on the machinery of dyadic grids and sparse operators
  used in the proof of the $A_2$ conjecture.
\end{abstract}

\section{Introduction}

The fractional integral operators, also called the Riesz
potentials, are the convolution operators
\[  I_\alpha f(x) = \int_\subRn \frac{f(y)}{|x-y|^{n-\alpha}}\,dy,
\qquad 0< \alpha<n.  \]
These operators are classical and for $1<p<\frac{n}{\alpha}$ and $q$ defined
by $\frac{1}{p}-\frac{1}{q} = \frac{\alpha}{n}$,  satisfy $I_\alpha : L^p \rightarrow L^q$.
When $p=1$ they satisfy the endpoint estimate  $I_\alpha : L^1
\rightarrow L^{q,\infty}$.
(Cf.~Stein~\cite{MR0290095}.)   One weight norm inequalities for these
operators were first considered by Muckenhoupt and
Wheeden~\cite{muckenhoupt-wheeden74}, who introduced the governing
class of weights, $A_{p,q}$.  For $1<p<\frac{n}{\alpha}$ and $q$ such that
$\frac{1}{p}-\frac{1}{q} = \frac{\alpha}{n}$, a weight (i.e., a non-negative, locally
integrable function) $w$ is in $A_{p,q}$ if
\[  [w]_{A_{p,q}} = \sup_Q \left(\avgint_Q
  w^{q}\,dx\right)^{\frac{1}{q}}
\left(\avgint_Q w^{-p'}\,dx\right)^{\frac{1}{p'}}<\infty, \]
where the supremum is taken over all cubes $Q$ in $\R^n$.  When $p=1$,
$q=\frac{n}{n-\alpha}$, we say $w\in A_{1,q}$ if 
\[  [w]_{A_{1,q}} = \sup_Q \esssup_{x\in Q} \left(\avgint_Q
  w^{q}\,dx\right)^{\frac{1}{q}} w(x)^{-1} < \infty. \]
Muckenhoupt and Wheeden showed that when $p>1$, $I_\alpha : L^p(w^p)\rightarrow
L^q(w^q)$ if and only $w\in A_{p,q}$, and when $p=1$, $I_\alpha : L^1(w)\rightarrow
L^{q,\infty}(w^q)$ when $w\in A_{1,q}$.   Their proof used a good-$\lambda$
inequality relating $I_\alpha$ and the fractional maximal operator,
\[ M_\alpha f(x) = \sup_Q |Q|^{\frac{\alpha}{n}}\avgint_Q |f(y)|\,dy \cdot \chi_Q(x). \]
Weighted norm inequalities for $M_\alpha$ were proved by generalizing
the earlier results for the Hardy-Littlewood maximal operator.
A different proof of the strong type inequality was given
in~\cite{cruz-uribe-martell-perez06b}:  there they used Rubio de Francia
extrapolation to prove a norm inequality relating $I_\alpha$ and
$M_\alpha$.  

Given $b\in BMO$ we define the commutator of a fractional integral by
\[ [b,I_\alpha] f(x) = b(x)I_\alpha f(x) - I_\alpha (bf)(x) 
= \int_\subRn \big( b(x) - b(y) \big)
\frac{f(y)}{|x-y|^{n-\alpha}}\,dy. \]
These commutators were introduced by Chanillo~\cite{MR642611}, who proved
that with $p$ and $q$ defined as above, $[b,I_\alpha] : L^p
\rightarrow L^q$.  He also proved that when $p$ is an even integer,
$b\in BMO$ is a necessary condition.  (The necessity for the full range
of $p$ was
recently shown by Chaffee~\cite{2014arXiv1410.4587C}.)   Commutators
are more singular than the fractional integral operator:  this can be
seen by the fact that when $p=1$, the do not map $L^1$ into
$L^{q,\infty}$.  For a counter-example and substitute endpoint estimate,
see~\cite{cruz-uribe-fiorenza03}.    In this paper it was also shown
that the strong type inequality is governed by $A_{p,q}$ weights:
if $1<p<\frac{n}{\alpha}$ and $w\in A_{p,q}$, then $[b,I_\alpha] : L^p(w^p)
\rightarrow L^q(w^q)$.  This proof relied on a sharp maximal function
estimate relating the commutator, $I_\alpha$, and $M_\alpha$.  A
different proof using extrapolation to relate the commutator to an
Orlicz fractional maximal operator was given
in~\cite{cruz-uribe-martell-perez06b}.  Yet another proof, one that gave
the sharp constant in terms of the $[w]_{A_{p,q}}$ characteristic, was
given in~\cite{cruz-moen2012}.  This proof used a Cauchy integral
formula argument due to Chung, {\em et al.}~\cite{MR2869172}.  

\medskip

In this paper we give new and elementary proofs of the one weight
inequalities for fractional integral operators and commutators.
More precisely, we prove the following three theorems.

\begin{theorem} \label{thm:weak-frac}
Given $0<\alpha<n$, let $q=\frac{n}{n-\alpha}$.  If $w\in A_{1,q}$,
then for all $f\in L^1(w)$,
\[ \sup_{t>0} t\, w^q(\{ x\in \R^n : |I_\alpha f(x)|>t \})^{\frac{1}{q}}
\leq C(n,\alpha)[w]_{A_{1,q}}^{1+q}\int_\subRn |f(x)| w(x)\,dx. \]
\end{theorem}

\begin{theorem} \label{thm:strong-frac-one}
Given $0<\alpha<n$, $1<p<\frac{n}{\alpha}$, let $q$ be  such that
$\frac{1}{p}-\frac{1}{q}=\frac{\alpha}{n}$.  If $w\in
A_{p,q}$, then for all $f\in L^p(w^p)$,
\[ \left(\int_\subRn |I_\alpha f(x)|^q w(x)^q\,dx\right) ^{\frac{1}{q}}
\leq C(n,p,\alpha) [w]_{A_{p,q}}^{1+\frac{q}{p'}+\frac{p'}{p}}
\left(\int_\subRn |f(x)|^pw(x)^p\,dx\right)^{\frac{1}{p}}. \]
\end{theorem}

\begin{theorem} \label{thm:strong-commutator}
Given $0<\alpha<n$, $1<p<\frac{n}{\alpha}$, let $q$ be such that
$\frac{1}{p}-\frac{1}{q}=\frac{\alpha}{n}$.  If $w\in
A_{p,q}$, then for all $f\in L^p(w^p)$ and $b\in BMO$,
\begin{multline*}
 \left(\int_\subRn |[b,I_\alpha]f(x)|^q w(x)^q\,dx\right)
 ^{\frac{1}{q}} \\
\leq C(n,p,\alpha) [w]_{A_{p,q}}^{\max(p',q)+1+\frac{q}{p'}+\frac{p'}{p}}\|b\|_{BMO}
\left(\int_\subRn |f(x)|^pw(x)^p\,dx\right)^{\frac{1}{p}}. 
\end{multline*}
\end{theorem}

We prove Theorems~\ref{thm:weak-frac}--\ref{thm:strong-commutator}
using the machinery of dyadic grids and sparse
operators.  Dyadic 
fractional integral operators date back to the work of
Sawyer and Wheeden~\cite{MR1175693}.  More recently, using the
machinery developed as part of the proof of the $A_2$ conjecture
for singular integral operators (see~\cite{hytonenP,MR3085756} and
the references they contain) dyadic fractional integral operators were
further developed and
applied to commutators in~\cite{cruz-moen2012,MR3065302,MR3224572}. 
The advantage of this approach is its simplicity:   it
avoids extrapolation, good-$\lambda$ inequalities and comparisons to
the fractional maximal operator.    One weakness of our proofs is that
they do not give sharp
dependence on the $A_{p,q}$ characteristic of the weights: this is
to be expected since we freely use their properties to simplify the
proofs, whereas any sharp constant proof must be arranged to use their
properties as few times as possible.  Sharp constants for the
fractional integral operator are given in~\cite{MR2652182}, and for
commutators in~\cite{cruz-moen2012}.

\medskip

The remainder of this paper is organized as follows.  In
Section~\ref{section:dyadic} we give some preliminary results about
dyadic grids, sparse operators, weighted fractional maximal operators,
and weights.  In Section~\ref{section:endpoint} we prove
Theorem~\ref{thm:weak-frac}.  Our proof adapts ideas first used by
Sawyer to prove two weight weak $(p,q)$ inequalities for fractional
integrals~\cite{MR719674}.  In Section~\ref{section:strong-frac} we
prove Theorem~\ref{thm:strong-frac-one}.  Our proof uses ideas of
P\'erez~\cite{MR1291534} from his proof of two weight inequalities for
fractional integrals,
and from the elementary proof of one weight inequalities for the
Hardy-Littlewood maximal operator due to Christ and
Fefferman~\cite{MR684636}.  In Section~\ref{section:strong-commutator}
we prove Theorem~\ref{thm:strong-commutator}.  Our proof uses
some ideas from the proof of two weight results
in~\cite{cruz-moen2012} to reduce the problem to an estimate
essentially the same as the one for the fractional integral in the
previous section.  Finally, in Section~\ref{section:cora} we give some
personal recollections about the late Cora Sadosky.

Throughout this paper notation is standard or will be defined as
needed.  The constant $n$ will always denote the dimension.  We will
denote constants by $C,\,c$, etc. and their value may change at each
appearance.   Unless otherwise specified, we will assume that
constants can depend on $p$, $\alpha$ and $n$ but we will keep track
of the dependence on the $A_{p,q}$ characteristic explicitly.

\section{Preliminary results}
\label{section:dyadic}

\subsection*{Dyadic grids and operators}  
We begin by defining dyadic grids and the dyadic fractional
integral operators.   Unless otherwise noted, the results given here
are taken from~\cite{cruz-antequera,cruz-moen2012,MR3065302,MR3224572}.

\begin{definition} \label{defn:dyadic-grid}
A collection of cubes $\D$ in $\R^n$ is a dyadic grid if provided that:
\begin{enumerate}
\setlength{\itemsep}{6pt}

\item If $Q\in \D$, then $\ell(Q)=2^k$ for some $k\in \Z$.

\item If $P,\,Q\in \D$, then $P\cap Q \in \{ P, Q, \emptyset \}$.

\item For every $k\in \Z$, the cubes $\D_k = \{ Q\in \D : \ell(Q)= 2^k
  \}$ form a partition of $\R^n$. 

\end{enumerate}
\end{definition}

\begin{definition} \label{defn:sparse}
Given a dyadic grid $\D$, a set $\Ss\subset \D$ is sparse if for every
$Q\in S$,
\[  \bigg|\bigcup_{\substack{P\in S\\P\subsetneq Q}} P\bigg|
\leq \frac{1}{2}|Q|. \]
Equivalently, if we define 
\[ E(Q) = Q \setminus \bigcup_{\substack{P\in S\\P\subsetneq Q}}
P, \]
then the sets $E(Q)$ are pairwise disjoint and $|E(Q)|\geq \frac{1}{2}|Q|$. 
\end{definition}

The classic example of a dyadic grid and sparse families are the
standard dyadic grid on $\R^n$ and the Calder\'on-Zygmund cubes
associated with an $L^1$ function.  See~\cite[Appendix~A]{MR2797562}. 

We next define a dyadic version of the fractional integral
operator and show that it can be used to bound $I_\alpha$ pointwise.   For
$f\in L^1_{loc}$ and a cube $Q$, let 
\[ \avg{f}_Q = \avgint_Q f(x)\,dx = \frac{1}{|Q|}\int_Q f(x)\,dy.  \]
Given $0<\alpha<n$ and a dyadic grid $\D$, define
\[ I_\alpha^\D f(x) = \sum_{Q\in \D} 
|Q|^{\frac{\alpha}{n}}\avg{f}_Q \cdot \chi_Q(x).  \]
Similarly, given a sparse subset $\Ss \subset \D$, define
\[ I_\alpha^\Ss f(x) = \sum_{Q\in \Ss} 
|Q|^{\frac{\alpha}{n}}\avg{f}_Q \cdot \chi_Q(x). \]

\begin{lemma} \label{lemma:dyadic-frac}
There exists a collection $\{\D_i\}_{i=1}^N$ of dyadic grids such that
for $0<\alpha<n$ and every non-negative function $f$,
\[ I_\alpha f(x) \leq C \sup_i I^{\D_i}_\alpha f(x). \]
Moreover, given any dyadic grid $\D$ and a non-negative function $f\in
L_c^\infty$, there exists a sparse set $\Ss\subset \D$ such that
\[ I^\D f(x) \leq C I^\Ss f(x). \]
\end{lemma}

Given $b\in BMO$ and $0<\alpha<n$, for any dyadic grid $\D$, define
the dyadic commutator 
\[ C_b^\D f(x) = \sum_{Q\in \D} |Q|^{\frac{\alpha}{n}}\avgint_Q |b(x)-b(y)|f(y)\,dy \cdot \chi_Q(x). \]

\begin{lemma} \label{lemma:dyadic-commutator}
There exists a collection $\{\D_i\}_{i=1}^N$ of dyadic grids such that
for $0<\alpha<n$, every non-negative function $f$, and every $b\in
BMO$,
\[ |[b,I_\alpha]f(x)| \leq C \sup_i  C_b^{\D^i} f(x). \]
\end{lemma}

\begin{remark} \label{remark:reduce}
It follows at once from Lemmas~\ref{lemma:dyadic-frac}
and~\ref{lemma:dyadic-commutator} that to prove norm inequalities for $I_\alpha$ and
$[b,I_\alpha]$ it will suffice to prove them for their dyadic
counterparts.   Moreover, since these dyadic integral operators are
positive, we may assume that $f$ is non-negative in our proofs.
Finally, by Fatou's lemma it will suffice to prove our results for
functions $f\in L^\infty_c$.  In particular, this will let us pass to
sparse operators.
\end{remark}

\subsection*{Weighted fractional maximal operators}

We begin with some basic facts about Orlicz spaces.  Some of these
will also be needed in Section~\ref{section:strong-commutator} below.
For further information on these spaces, see~\cite{rao-ren91}; for
their use in weighted norm inequalities, see~\cite{MR2797562}.  Given
a weight $\sigma$, let $d\sigma = \sigma\,dx$.  We define averages 
with respect to the measure $d\sigma$:
\[ \avg{f}_{Q,\sigma} =
 \avgint_Q f(x)\,d\sigma = \frac{1}{\sigma(Q)}\int_Q f(x)\,d\sigma. \]
Given a Young function $\Phi$ and a cube $Q$,
 define the normalized Luxemburg norm with respect to $\Phi$ and $d\sigma$ by
\[ \|f\|_{\Phi, Q,\sigma} = \inf\bigg\{ \lambda > 0 :
\avgint_Q \Phi\bigg(\frac{|f(x)|}{\lambda}\bigg)\,d\sigma \leq 1
\bigg\}. \]
If we let $\Phi(t)=t^p$, $1\leq p<\infty$, then 
\[ \|f\|_{\Phi,Q,\sigma} = \left(\avgint_Q
  |f(x)|^p\,d\sigma\right)^{\frac{1}{p}} = \|f\|_{p,Q,\sigma}. \]
Associated to any Young function is its associate function
$\bar{\Phi}$.  We have the generalized H\"older's inequality:  for any
cube $Q$,
\[ \avgint_Q |f(x)g(x)|\,d\sigma \leq
C\|f\|_{\Phi,Q,\sigma}\|g\|_{\bar{\Phi},Q,\sigma}; \]
the constant depends only on $\Phi$.  

Hereafter, we will let $\Phi(t)=t\log(e+t)$; then it can be shown that
$\bar{\Phi}(t) \approx e^t-1$.   It follows for this choice of $\Phi$
that for $1<p<\infty$,
\[ \|f\|_{1,Q,\sigma} \leq \|f\|_{\Phi,Q,\sigma} \leq C(p)\|f\|_{p,Q,\sigma}. \]

We now define a weighted dyadic fractional maximal operator.  Given a
dyadic grid $\D$ and a weight $\sigma$, 
for $0\leq \alpha <n$ define 
\[ M^\D_{\Phi,\sigma,\alpha} f(x) =\sup_{Q\in \D} \sigma(Q)^{\frac{\alpha}{n}}
\|f\|_{\Phi,Q,\sigma}.  \]
If
$\Phi(t)=t$ we write $M^\D_{\sigma,\alpha}$.    If $\alpha=0$, then we
simply write $M^\D_{\Phi,\sigma}$ or $M^\D_\sigma$ if $\Phi(t)=t$.

\begin{lemma} \label{lemma:max-op}
Let $\Phi(t)=t\log(e+t)$.  Given $1<p<\infty$ and $0\leq \alpha <n$, define $q$ by
$\frac{1}{p}-\frac{1}{q}=\frac{\alpha}{n}$.   Given a weight $\sigma$ and a dyadic grid $\D$,
$M^\D_{\Phi,\sigma,\alpha} : L^p(\sigma) \rightarrow L^q(\sigma)$.    The same inequality holds for $M^\D_{\sigma,\alpha}$.
\end{lemma}

\begin{proof}
This result is well-known when $\Phi(t)=t$; the proof is essentially
the same for $\Phi(t)=t\log(e+t)$  and we sketch the details.  
By off-diagonal Marcinkiewicz interpolation (see~\cite{stein-weiss71})
it will suffice to prove the corresponding weak $(p,q)$ inequality:
\[ \sigma(\{ x \in \R^n : M^\D_{\Phi,\sigma,\alpha} f(x) > t\})^{\frac{1}{q}}
\leq \frac{C}{t} \left(\int_\subRn |f(x)|^p\,d\sigma\right)^{\frac{1}{p}}. \]
 Fix $t>0$; then we can decompose
the level set as the union of disjoint cubes $Q \in \Q_t \subset \D$ that satisfy
\[ \sigma(Q)^{\frac{\alpha}{n}}
\|f\|_{\Phi,Q,\sigma} > t. \]
Therefore, since the cubes in $\Q_t$ are disjoint and $q/p\geq 1$, we have that
\begin{align*}
 \sigma(\{ x \in \R^n : M^\D_{\Phi,\sigma,\alpha} f(x) > t\})
& = \sum_{Q\in \Q_t}  \sigma(Q) \\
& \leq t^{-q} \sum_{Q\in \Q_t} \sigma(Q)^{1+q\frac{\alpha}{n}}
  \|f\|_{\Phi,Q,\sigma}^q \\
& \leq  \frac{C}{t^q}\sum_{Q\in \Q_t}  \sigma(Q)^{1+q\frac{\alpha}{n}} \left(\avgint_{Q}
  |f|^p\,d\sigma\right)^{\frac{q}{p}} \\
& \leq \frac{C}{t^q} \sum_{Q\in \Q_t} \left(\int_{Q} |f|^p\,d\sigma\right)^{\frac{q}{p}} \\
& \leq \frac{C}{t^q}\left(\sum_{Q\in \Q_t}  \int_{Q} |f|^p\,d\sigma\right)^{\frac{q}{p}} \\
& \leq \frac{C}{t^q} \left(\int_\subRn |f(x)|^p\,d\sigma\right)^{\frac{q}{p}}. 
\end{align*}
\end{proof}

\subsection*{Properties of $A_{p,q}$ weights}
In this section we gather a few basic facts about the $A_{p,q}$
weights and the closely related Muckenhoupt $A_p$ weights; for further
information see~\cite{duoandikoetxea01}.  For
$1<p<\infty$, $w\in A_p$ if
\[ [w]_{A_p} = \sup_Q \avgint_Q w(x)\,dx \left(\avgint_Q
  w(x)^{1-p'}\,dx\right)^{p-1} < \infty. \]
For $p=1$, $w\in A_1$ if
\[ [w]_{A_1} = \sup_Q \esssup_{x\in Q} \left(\avgint{Q}
  w(x)\,dx\right)w(x)^{-1} < \infty. \]
It follows at once from the definition that for all $p$ and $q$ such
that $\frac{1}{p}-\frac{1}{q}=\frac{\alpha}{n}$,  then
$w\in A_{p,q}$ if and only if $w^q \in A_r$, $r=1+\frac{q}{p'}$ and
$[w^q]_{A_r}= [w]_{A_{p,q}}^q$.  (When $p=1$, we interpret
$\frac{q}{p'}$ as $0$.)  By the duality of $A_p$ weights, $w^{-p'}\in
A_{r'}$ and $[w^{-p'}]_{A_{r'}}=[w]_{A_{p,q}}^{p'}$.  

As a consequence, we have that both $w^q$ and $w^{-p'}$ are in
$A_\infty$.  Below we will need to use two properties of $A_\infty$
weights; for both we give the sharp constant version.  There are
multiple definitions of the $A_\infty$ characteristic (cf.~\cite{duo-mr-ombrosi}) but for our
purposes we will simply use the fact that $[w]_{A_\infty} \leq
C(n)[w]_{A_p}$.  

\begin{lemma} \label{lemma-Ainfty}
Given a weight $\sigma \in A_p \subset A_\infty$, then for any cube
$Q$ and set $E\subset Q$:
\begin{enumerate}

\item $\displaystyle  \left(\frac{|E|}{|Q|}\right)^p \leq [\sigma]_{A_p}
  \frac{\sigma(E)}{\sigma(Q)}$;

\item $\displaystyle  \frac{\sigma(E)}{\sigma(Q)} \leq
  2\left(\frac{|E|}{|Q|}\right)^{\frac{1}{s'}}$,
where $s' = c(n) [\sigma]_{A_\infty}$. 

\end{enumerate}
\end{lemma}

\begin{proof}
The first inequality follows from the definition of $A_p$:
see~\cite{duoandikoetxea01}.   The second follows from the sharp form
of the reverse H\"older inequality due to Hyt\"onen and
P\'erez~\cite{MR3092729}:  for every cube $Q$,
\[ \left(\avgint_Q \sigma(x)^s\,dx\right)^{\frac{1}{s}} \leq 2
\avgint_Q \sigma(x)\,dx, \]
where $s = 1+ \frac{1}{c(n)[\sigma]_{A_\infty}}$.  The desired
inequality follows if we apply H\"older's inequality with exponent
$s$ to $\avgint_Q \sigma(x)\chi_E(x)\,dx$.
\end{proof}

\section{Proof of Theorem~\ref{thm:weak-frac}}
\label{section:endpoint}

  By Remark~\ref{remark:reduce} it will suffice to prove the weak
  $(1,q)$ inequality
  for the dyadic operator $I_\alpha^\D$, where $\D$ is an arbitrary
  dyadic grid, and for $f$ a non-negative function in $L^\infty_c$.
Furthermore, since $\supp(f)$ is contained in the union of at most $2^n$
  cubes in $\D$ (of equal size), by linearity it will suffice to prove
  this inequality assuming $\supp(f)$ is contained in a single cube.

Given a cube $Q$, we decompose $I_\alpha^\D$ into its inner and outer parts:
\begin{align*}
I_\alpha^\D f(x) = \sum_{P\in \D} |P|^{\frac{\alpha}{n}} \avg{f}_P \chi_P(x)
& = \sum_{\substack{P\in \D \\ P\subseteq Q}} |P|^{\frac{\alpha}{n}} \avg{f}_P \chi_P(x) + 
\sum_{\substack{P\in \D \\ Q\subsetneq P}} |P|^{\frac{\alpha}{n}} \avg{f}_P \chi_P(x) \\
& = I_{\alpha,Q}^{\D,in}f(x) + I_{\alpha,Q}^{\D,out}f(x). 
\end{align*}
For each $t>0$ define the set 
\[ E_t = \{ x\in \R^n : I_\alpha^\D f(x) > t \}.  \]
We can write $E_t$ as the union of a collection $\Q_t$ of maximal,
disjoint cubes in $\D$.  By maximality, if $Q\in \Q_t$ and we let $\hat{Q}$ be its dyadic parent, then there exists $z\in \hat{Q}\setminus Q$ such that for every $x\in Q$, 
\[ t \geq I_\alpha^D f(z) \geq I_{\alpha,Q}^{\D,out}f(z) =  I_{\alpha,Q}^{\D,out}f(x). \]
In particular, if we take any $x \in Q\cap E_{2t}$, then we must have that 
\[ I_{\alpha,Q}^{\D,in}f(x) > t.  \]

Now fix $T>0$ and let $t$ be such that $0<t<T/2$.   Let $v=w^q$.
Partition the cubes in $\Q_t$ into two sets:  we say $Q\in \Q_{L}$
(large cubes) if
\begin{equation} \label{eqn:large-cubes}
 v(Q\cap E_{2t}) \geq  2^{-q-1} v(Q),  
\end{equation}
and if the reverse inequality holds we say $Q\in \Q_{S}$ (small cubes).  
Since $E_{2t}\subset E_t$,  we have that
\begin{align*}
 (2t)^q v(E_{2t}) 
& = (2t)^q \sum_{Q\in \Q_t} v(Q\cap E_{2t}) \\
& = (2t)^q \sum_{Q\in \Q_S} v(Q\cap E_{2t}) 
+ (2t)^q \sum_{Q\in \Q_L} v(Q\cap E_{2t}) \\
& = I_S + I_L.  
\end{align*}

We estimate each sum separately.  The first is straightforward:
\[ I_S \leq (2t)^q \sum_{Q\in \Q_S} 2^{-q-1}v(Q) = \frac{1}{2}t^q v(E_t)
\leq \frac{1}{2}\sup_{0<t<T/2} t^q v(E_t)\leq \frac{1}{2}\sup_{0<t<T} t^q v(E_t). \]

To estimate the second, by inequality~\eqref{eqn:large-cubes} and
duality (since $I_{\alpha,Q}^\D$ is self-adjoint) we have that
\begin{align*}
I_L 
& = (2t)^q \sum_{Q\in \Q_L} v(Q\cap E_t)^q v(Q\cap E_t)^{1-q} \\
& \leq (2t)^q \sum_{Q\in \Q_L} v(\{ x\in Q :  I_{\alpha,Q}^{\D,in}f(x) > t \})^q v(Q\cap E_t)^{1-q} \\
& \leq C(q)\sum_{Q\in \Q_L} \left(\int_Q  I_{\alpha,Q}^{\D,in}f(x) v(x)\,dx\right)^q v(Q)^{1-q} \\
& =  C(q) \sum_{Q\in \Q_L} \left(\int_Q  f(x) I_{\alpha,Q}^{\D,in}v(x)\,dx\right)^q v(Q)^{1-q}.
\end{align*}

By the definition of $I_{\alpha,Q}^{\D,in}$ and $q$,
\[  I_{\alpha,Q}^{\D,in}v(x) = 
\sum_{\substack{P\in \D \\ P\subseteq Q}} |P|^{\frac{\alpha}{n}} \avg{v}_P \chi_P(x) 
= \sum_{\substack{P\in \D \\ P\subseteq Q}} |P|^{\frac{1}{q'}} \avg{v}_P ^{\frac{1}{q'}}\avg{v}_P ^{\frac{1}{q}} \chi_P(x). \]
Since $w\in A_{1,q}$, $v\in A_1 \subset A_\infty$, and so for any $P\subseteq Q$
$\avg{v}_P ^{\frac{1}{q}} \leq [w]_{A_{1,q}}w(x)$, and by 
Lemma~\ref{lemma-Ainfty},
\[  v(P) \leq 2\left(\frac{|P|}{|Q|}\right)^{\frac{1}{s'}} v(Q), \]
where $s' \leq c(n) [w]_{1,q}^q$.    If we fix $x$, then there exists a
nested sequence of  dyadic cubes $P_k\subset Q$, $k\geq 0$,  such that $x\in P_k$ and $\ell(P_k)=2^{-k} \ell(Q)$.  
Therefore, 
\[ \sum_{\substack{P\in \D \\ P\subseteq Q}} |P|^{\frac{1}{q'}} \avg{v}_P ^{\frac{1}{q'}}\chi_P(x)
= \sum_{k=0}^\infty v(P_k) ^{\frac{1}{q'}} 
\leq 2^{\frac{1}{q'}} \sum_{k=0}^\infty 2^{-\frac{kn}{q's'}} v(Q) ^{\frac{1}{q'}}
\leq C\, s' v(Q) ^{\frac{1}{q'}}.  \]

If we combine all of these estimates we get that
\[  I_{\alpha,Q}^{\D,in}(v)(x) \lesssim [w]_{A_{1,q}}^{1+q} v(Q)^{\frac{1}{q'}} w(x). \]
Hence,
\begin{multline*}
 I_L \leq C [w]_{A_{1,q}}^{q(q+1)} \sum_{Q\in \Q_L} 
\left(\int_Q f(x) w(x)\,dx\right)^{q} v(Q) ^{\frac{q}{q'}+1-q} \\
\leq [w]_{A_{1,q}}^{q(q+1)} \left(\sum_{Q\in \Q_L} \int_Q f(x) w(x)\,dx\right)^{q} 
\leq [w]_{A_{1,q}}^{q(q+1)}\left(\int_\subRn f(x) w(x)\,dx\right)^{q}. 
\end{multline*}

If we combine this with the estimate for $I_S$, we get 
\[  (2t)^q v(E_{2t}) \leq \frac{1}{2}\sup_{0<t<T} t^q v(E_t) + C [w]_{A_{1,q}}^{q(q+1)}\left(\int_\subRn f(x) w(x)\,dx\right)^{q}. \]
Since this inequality holds for any $t$, $0<t<T/2$, on the left, we can take the supremum over these values of $t$ and rearrange terms to get 
\[ \sup_{0<t<T} t^q v(E_t)  \leq C [w]_{A_{1,q}}^{q(q+1)}\left(\int_\subRn f(x) w(x)\,dx\right)^{q}, \]
provided that the supremum  is finite.   The desired inequality would then follow if we let $T\rightarrow \infty$. 

To see that this supremum is finite: By assumption,
$\supp(f)\subset Q$ for some cube $Q\in \D$.  If $x\not\in Q$, then
$I_\alpha^\D f(x)\neq 0$, only if there exists a dyadic cube $P$ that
contains $x$ and $Q$.  Let $P_0$ be the smallest such cube, and let
$\{P_k\}$ be the sequence of dyadic cubes containing $x$ such that
$P_0\subset P_k$ and $\ell(P_k)=2^k\ell(P_0)$.  Then
\[
 I_\alpha^\D f(x) = \sum_{k=0}^\infty 2^{k(\alpha-n)}|P_0|^{\frac{\alpha}{n}-1} \int_{P_k} f(y)\,dy 
\leq C |P_0|^{\frac{\alpha}{n}-1} \int_{P_0} f(y)\,dy \leq C M_\alpha^\D f(x).  
\]
Therefore, we have that 
\[ \sup_{0<t<T} t^q v(E_t) \leq T^q v(Q) +  \sup_{0<t<T}t^q\,v(\{ x\in \R^n\setminus Q : M_\alpha^\D f(x)>t/C\}). \]
The first term is finite since $v$ is locally integrable, and the
second is finite by the weighted weak $(1,q)$ inequality for $M_\alpha^\D$
(cf.~\cite{muckenhoupt-wheeden74}).     This completes the proof.

\section{Proof of Theorem~\ref{thm:strong-frac-one}}
\label{section:strong-frac}

By Remark~\ref{remark:reduce}, it will suffice to prove the strong $(p,q)$
inequality for $f$ non-negative and in $L^\infty_c$.  We may also
replace $I_\alpha$ by the sparse operator
$I_\alpha^\Ss$, where $\Ss$ is any sparse
subset of a dyadic grid~$\D$.

Let $v=w^q$ and $\sigma=w^{-p'}$ and estimate as follows:  there exists $g\in
L^{q'}(w^{-q'})$, $\|gw^{-1}\|_{q'}=1$, such that
\begin{multline*}
\|(I_\alpha^\Ss f)w\|_q 
 = \int_\subRn I_\alpha f(x) g(x)\,dx 
 = \sum_{Q\in \Ss} |Q|^{\frac{\alpha}{n}}\avg{f}_Q \int_Q g(x)\,dx \\
 = \sum_{Q\in \Ss}  |Q|^{\frac{\alpha}{n}-1}\sigma(Q)
v(Q)^{1-\frac{\alpha}{n}}
\avg{f\sigma^{-1}}_{Q,\sigma}
v(Q)^{\frac{\alpha}{n}}\avg{gv^{-1}}_{Q,v}. 
\end{multline*}
Since $1-\frac{\alpha}{n}=\frac{1}{p'}+\frac{1}{q}$, by the definition
of the $A_{p,q}$ condition and Lemma~\ref{lemma-Ainfty}
(applied to both $v$ and $\sigma$) we have that
\begin{equation} \label{eqn:two-Ainfty}
 |Q|^{\frac{\alpha}{n}-1}\sigma(Q)
v(Q)^{1-\frac{\alpha}{n}}  \leq [w]_{A_{p,q}} \sigma(Q)^{\frac{1}{p}}
v(Q)^{\frac{1}{p'}} 
\leq  [w]_{A_{p,q}}^{1+\frac{q}{p'}+\frac{p'}{p}} \sigma(E(Q))^{\frac{1}{p}}
v(E(Q))^{\frac{1}{p'}}.
\end{equation}
If we combine these two estimates,  then by H\"older's
inequality and Lemma~\ref{lemma:max-op} we get that 
\begin{align*}
&\|(I_\alpha^\Ss f)w\|_q \\
& \quad \leq  C [w]_{A_{p,q}}^{1+\frac{q}{p'}+\frac{p'}{p}}\sum_{Q\in\Ss}
\avg{f\sigma^{-1}}_{Q,\sigma}\sigma(E(Q))^{\frac{1}{p}}
v(Q)^{\frac{\alpha}{n}}\avg{gv^{-1}}_{Q,v}v(E(Q))^{\frac{1}{p'}} \\
& \quad \leq C [w]_{A_{p,q}}^{1+\frac{q}{p'}+\frac{p'}{p}}\bigg(\sum_{Q\in \Ss} 
\avg{f\sigma^{-1}}_{Q,\sigma}^p \sigma(E(Q))\bigg)^{\frac{1}{p}}
\bigg(\sum_{Q\in \Ss} [v(Q)^{\frac{\alpha}{n}}\avg{gv^{-1}}_{Q,v}]^{p'}
  v(E(Q))\bigg)^{\frac{1}{p'}} \\
& \quad \leq C [w]_{A_{p,q}}^{1+\frac{q}{p'}+\frac{p'}{p}}\bigg(\sum_{Q\in \Ss} 
\int_{E(Q)} M_\sigma^\D (f\sigma^{-1})(x)^p \,d\sigma\bigg)^{\frac{1}{p}}
\bigg(\sum_{Q\in \Ss} 
\int_{E(Q)} M_{v,\alpha}^\D(gv^{-1})(x)^{p'}\,dv
\bigg)^{\frac{1}{p'}} \\
& \quad \leq  C [w]_{A_{p,q}}^{1+\frac{q}{p'}+\frac{p'}{p}}\bigg(
\int_\subRn M_\sigma^\D (f\sigma^{-1})(x)^p \,d\sigma\bigg)^{\frac{1}{p}}
\bigg(
\int_\subRn M_{v,\alpha}^\D(gv^{-1})(x)^{p'}\,dv
\bigg)^{\frac{1}{p'}} \\
& \quad \leq C [w]_{A_{p,q}}^{1+\frac{q}{p'}+\frac{p'}{p}}  \bigg(
\int_\subRn (f (x)\sigma (x)^{-1})^p \,d\sigma\bigg)^{\frac{1}{p}}
\left(
\int_\subRn (g (x) v (x)^{-1})^{q'}\,dv
\right)^{\frac{1}{q'}} \\
& \quad = C [w]_{A_{p,q}}^{1+\frac{q}{p'}+\frac{p'}{p}} \|fw\|_p \|gw^{-1}\|_{q'} \\
&  \quad = C [w]_{A_{p,q}}^{1+\frac{q}{p'}+\frac{p'}{p}} \|fw\|_p.
\end{align*}

\section{Proof of Theorem~\ref{thm:strong-commutator}}
\label{section:strong-commutator}

For our proof we need two lemmas.  The first is a weighted estimate for functions in $BMO$; our
proof adapts ideas from Ho~\cite{MR2849152}.  

\begin{lemma} \label{lemma:wtd-bmo}
Let $\Phi(t)=t\log(e+t)$.  Given a weight $\sigma \in A_\infty$, then for any $b\in BMO$ and any
cube $Q$,
\[ \|b-\avg{b}_Q\|_{\bar{\Phi},Q,\sigma} 
\leq C[\sigma]_{A_\infty}\|b\|_{BMO}. \]
\end{lemma}

\begin{proof}
By the John-Nirenberg inequality, there exist constants $C_1,\,C_2$
such that for every cube $Q$ and $\lambda,\, t>0$,
\[ |\{ x\in Q : |b(x)-\avg{b}_Q| > \lambda t \}|
\leq C_1|Q|\exp\left( -\frac{C_2 \lambda t }{\|b\|_{BMO}}\right). \]
Since $\sigma \in A_\infty$, by Lemma~\ref{lemma-Ainfty} we have that
\[ \sigma(\{ x\in Q : |b(x)-\avg{b}_Q| > \lambda t \})
\leq 2C_1^{\frac{1}{s'}} \sigma(Q)\exp\left( -\frac{C_2 \lambda t}{\|b\|_{BMO}s'}\right), \]
where $s'=c(n)[\sigma]_{A_\infty}$.  Let 
\[ \lambda =  \frac{(1+2kC_1^{\frac{1}{s'}})\|b\|_{BMO}s'}{C_2} =
C[\sigma]_{A_\infty}\|b\|_{BMO}, \]
where $\bar{\Phi}(t) \leq ke^t$.  Then  we have that
\begin{align*}
\avgint_Q
  \bar{\Phi}\bigg(\frac{|b(x)-\avg{b}_Q|}{\lambda}\bigg)\,d\sigma 
& \leq k\sigma(Q)^{-1} \int_0^\infty e^t \sigma(\{ x\in Q :
  |b(x)-\avg{b}_Q| > \lambda t \}) \,dt \\
& \leq  2kC_1^{\frac{1}{s'}}\int_0^\infty e^t \exp\left( -\frac{C_2 \lambda t
  }{\|b\|_{BMO}s'}\right)\,dt \\
& \leq  2kC_1^{\frac{1}{s'}} \int_0^\infty e^{- 2kC_1^{\frac{1}{s'}}t}\,dt \\
& = 1. 
\end{align*}
Therefore, by the definition of the Luxemburg norm, we get the desired
inequality.
\end{proof}

The second lemma is a weighted variant of an estimate
from~\cite{cruz-uribe-martell-perez06b}; when $\Phi(t)=t$ the
unweighted estimate is originally due to Sawyer and
Wheeden~\cite{MR1175693}.

\begin{lemma} \label{lemma:sum-alpha}
Fix $0<\alpha<n$,  a dyadic grid $\D$, a weight $\sigma$ and a Young function $\Phi$.  Then for any $P\in
\D$ and any function $f$,
\[ \sum_{\substack{Q\in \D \\ Q\subset P}} 
|Q|^{\frac{\alpha}{n}} \sigma(Q)|\|f\|_{\Phi,Q,\sigma}
\leq C(\alpha) |P|^{\frac{\alpha}{n}}  \sigma(P)|\|f\|_{\Phi,P,\sigma}. \]
\end{lemma}

\begin{proof}
To prove this we need to replace the  Luxemburg norm with the equivalent Amemiya norm~\cite[Section~3.3]{rao-ren91}:  
\[ \|f\|_{\Phi,P,\sigma} \leq  \inf_{\lambda>0} \left\{ \lambda
  \avgint_P 1+ \Phi\left(\frac{|f(x)|}{\lambda}\right)\,d\sigma \right\} \leq 2\|f\|_{\Phi,P,\sigma}. \]
By the second inequality, we can fix $\lambda_0>0$ such that the
middle quantity is less than
$3\|f\|_{\Phi, P,\sigma}$.   Then by the first inequality, 
\begin{align*}
 \sum_{\substack{Q\in \D \\ Q\subset P}} 
|Q|^{\frac{\alpha}{n}} |Q|\|f\|_{\Phi,Q,\sigma}
& = \sum_{k=0}^\infty \sum_{\substack{Q\subset P\\
  \ell(Q)=2^{-k}\ell(P)}}
|Q|^{\frac{\alpha}{n}} |Q|\|f\|_{\Phi,Q,\sigma} \\
& \leq |P|^{\frac{\alpha}{n}} \sum_{k=0}^\infty  2^{-k\alpha} 
\sum_{\substack{Q\subset P\\
  \ell(Q)=2^{-k}\ell(P)}} 
\lambda_0 \int_Q 1+ \Phi\left(\frac{|f(x)|}{\lambda_0}\right)\,d\sigma
\\
& = C |P|^{\frac{\alpha}{n}}
\lambda_0 \int_P 1+ \Phi\left(\frac{|f(x)|}{\lambda_0}\right)\,d\sigma
\\
& \leq C |P|^{\frac{\alpha}{n}}\sigma(P)\|f\|_{\Phi,P,\sigma}.
\end{align*}
\end{proof}

\begin{proof}[Proof of Theorem~\ref{thm:strong-commutator}]
Again by Remark~\ref{remark:reduce}, it will suffice to prove that given any dyadic grid $\D$,
\[ \|C^\D_b f\|_{L^q(w^q)} \leq C\|f\|_{L^p(w^p)}, \]
where  $f$ is non-negative
and in $L^\infty_c$.   By duality there exists a non-negative function
$g\in L^{q'}(w^{-q})$, $\|g\|_{L^{q'}(w^{-q}) }=1$, such that 
\begin{align*}
\|C^\D_b f\|_{L^q(w^q)}
& = \int_\subRn C^\D_b f(x)g(x)\,dx \\
& = \sum_{Q\in \D} |Q|^{\frac{\alpha}{n}} \int_Q\avgint_Q |b(x)-b(y)|f(y) g(x)
  \,dy\,dx \\
& \leq \sum_{Q\in \D} |Q|^{\frac{\alpha}{n}} \avgint_Q
  |b(y)-\avg{b}_Q|f(y)\,dy \int_Q g(x)\,dx \\
& \qquad \qquad + \sum_{Q\in \D} |Q|^{\frac{\alpha}{n}} \avgint_Q
  |b(x)-\avg{b}_Q | g(x)\,dx \int_Qf(y)\,dy \\
& = I_1 + I_2. 
\end{align*}

We will first estimate $I_1$.  
Let $v=w^{q}$ and $\sigma=w^{-p'}$, and let $\Phi(t)=t\log(e+t)$.  Then by
Lemma~\ref{lemma:wtd-bmo}, since $[\sigma]_{A_\infty}\leq [w]^{p'}_{A_{p,q}}$, 
\begin{align*}
I_1
& = \sum_{Q\in \D} |Q|^{\frac{\alpha}{n}} \sigma(Q)\avgint_Q
  |b(y)-\avg{b}_Q|f(y)\sigma(y)^{-1}\,d\sigma \avgint_Q g(x)\,dx \\
& \leq C \sum_{Q\in \D} |Q|^{\frac{\alpha}{n}} \sigma(Q)
\|f\sigma^{-1}\|_{\Phi,Q,\sigma} \|b-\avg{b}_Q\|_{\bar{\Phi},Q,\sigma}
  \avgint_Q g(x)\,dx \\
& \leq C [w]^{p'}_{A_{p,q}}\|b\|_{BMO}
\sum_{Q\in \D} |Q|^{\frac{\alpha}{n}} \sigma(Q)
  \|f\sigma^{-1}\|_{\Phi,Q,\sigma}\avgint_Q g(x)\,dx. 
\end{align*}

We now want to show that we can replace the summation over cubes in
$\D$ by a summation over a sparse subset $\Ss$ of $\D$.  We do
this using an argument from~\cite{cruz-uribe-martell-perez06b}; see
also~\cite{cruz-moen2012}. 
Fix $a=2^{n+1}$ and define the sets
\[ \Omega_k = \{ x\in \R^n : M^\D g(x)>a^k\}.  \]
Then arguing exactly as in the construction of the Calder\'on-Zygmund
cubes (see~\cite[Appendix~A]{MR2797562}), each set $\Omega_k$ is the
union of a collection $\Ss_k$ of maximal, disjoint cubes in $\D$ that
have the property that $a^k < \avg{g}_Q \leq 2^na^k$.  Moreover, the
set $\Ss= \bigcup_k \Ss_k$ is sparse.  

Now let 
\[ \C_k = \{ Q\in \D : a^k < \avg{g}_Q \leq a^{k+1} \}. \]
Then by the maximality of the cubes in $\Ss_k$, every cube $P\in \C_k$
is contained in a unique cube in $\Ss_k$.   Therefore, we can continue
the above estimate and apply Lemma~\ref{lemma:sum-alpha} to get
\begin{align*}
I_1 
& \leq  C [w]^{p'}_{A_{p,q}}\|b\|_{BMO}
\sum_k \sum_{Q\in \C_k} |Q|^{\frac{\alpha}{n}} \sigma(Q)
  \|f\sigma^{-1}\|_{\Phi,Q,\sigma}\avgint_Q g(x)\,dx \\
&\leq  C [w]^{p'}_{A_{p,q}}\|b\|_{BMO}
\sum_k a^k \sum_{P\in \Ss_k} \sum_{Q\subset P} |Q|^{\frac{\alpha}{n}} \sigma(Q)
  \|f\sigma^{-1}\|_{\Phi,Q,\sigma} \\
& \leq  C [w]^{p'}_{A_{p,q}}\|b\|_{BMO}
\sum_k  \sum_{P\in \Ss_k} |P|^{\frac{\alpha}{n}} \sigma(P)
  \|f\sigma^{-1}\|_{\Phi,P,\sigma} \avgint_P g(x)\,dx \\
& = C [w]^{p'}_{A_{p,q}}\|b\|_{BMO}
\sum_{P\in \Ss} |P|^{\frac{\alpha}{n}-1} \sigma(P)v(P)^{1-\frac{\alpha}{n}}
  \|f\sigma^{-1}\|_{\Phi,Q,\sigma}
v(P)^{\frac{\alpha}{n}}\avgint_Q gv^{-1}\,dv.
\end{align*}

We can now argue exactly as in the proof of
Theorem~\ref{thm:strong-frac-one}, beginning with the
estimate~\eqref{eqn:two-Ainfty} and applying Lemma~\ref{lemma:max-op}
to complete the estimate of $I_1$  with a constant
$[w]_{A_{p,q}}^{p'+1+\frac{q}{p'}+\frac{p'}{p}}$.    

The estimate for $I_2$ is
essentially the same, exchanging the roles of $f$ and $g$ and $\sigma$
and $v$.  This yields the above estimate except that the constant
is now
$[w]_{A_{p,q}}^{q+1+\frac{q}{p'}+\frac{p'}{p}}$. 
 This completes the proof.
\end{proof}

\section{Tres recuerdos de Cora Sadosky}
\label{section:cora}

I first met Cora at an AMS sectional meeting in Burlington, Vermont,
in 1995.  It began inauspiciously: at the reception
on the first night, a determined looking woman came up to me, waved her
finger under my nose and said, ``I have a bone to pick with you.  We
will talk later,'' and then marched off.   She found me again about 30
minutes later and proceeded to explain.  The year before
she had been asked by an NSF reviewer for her opinion of my proposal
which mentioned the two weight problem for the Hilbert transform.
She had told the reviewer to refer me to a paper by her and Mischa Cotlar where they gave the first
(and for a long time the only) characterization of these pairs of
weights.  He did not, however, share this reference, and when I
published the paper based on this work I did not cite it.  Cora assumed that I had
simply disregarded this advice and was understandably annoyed.  
However, once I explained that I had never received this information she
immediately became much friendlier and invited me to visit her in
Washington, D.C.  

For the next few years she took an interest in my career.  Her first
major intervention on my behalf came in 1996, when she 
applied her forceful personality, first to convince me that I must
attend the {\em El Escorial} conference in 1996 (despite moving,
changing jobs, and having two small children and a pregnant wife), and
then to strong-arm funding from a colleague to pay for my trip.  It
was at this meeting that I met, among others, Carlos P\'erez, and
began a collaboration that has continued to
the present day.

Two years later, in 1998, we met again at an AMS sectional meeting in
Albuquerque.  At this meeting she picked up on a point that my Spanish
colleagues were also making:  given my
name and my ancestry, I really ought to be able to speak Spanish.  Her solution was that
I should ``read a good math book in Spanish.''  She strongly
recommended that I read Javier Duoandikoetxea's book, {\em An\'alisis
  de Fourier}, telling me that I would see some good mathematics as
well as ``learn Spanish.''  For the next year I worked through the text
line by line, in the process writing a complete translation.  I
approached Javi
with an offer to complete the translation and update
the notes, and together we produced an
English edition.  In the process I did in fact learn a
great deal of harmonic analysis, but unfortunately, Cora's original
goal was not achieved:  my spoken Spanish did not improve appreciably.  
Moreover, this translation had the unintended consequence of convincing large
numbers of mathematicians from Spain, Argentina and elsewhere that I
did in fact speak Spanish.  

Cora and I never collaborated on a paper.  She suggested several
projects, but my interests were moving away from hers and nothing came
to fruition.  At the time I never really quite understood or
appreciated the support she provided at these points in my career, and
it is only in looking back that I realize how much I owe her.  So
belatedly I say, {\em much\'\i simas gracias, Cora.}

\bibliographystyle{plain}
\bibliography{simple-frac}

\end{document}